\documentclass[12pt]{amsart}
\usepackage[dvipsnames]{xcolor}
\usepackage[]{amsmath, amsthm, amsfonts, verbatim, amssymb, tikz}

\usepackage[pagebackref]{hyperref}
\hypersetup{
	colorlinks=true,
    linkcolor=red,
    filecolor=blue,      
    urlcolor=blue,
    citecolor=ForestGreen,
    linktoc=page}
\usepackage[alphabetic]{amsrefs}
\usepackage{rotating}
\usepackage{epstopdf}
\usepackage{pifont}
\usepackage{enumerate}
\usepackage{graphicx}
\usepackage{floatrow}

\DeclareMathAlphabet{\mathpzc}{OT1}{pzc}{m}{it}
\usepackage{tikz,tikz-cd, color}
\usepackage[makeroom]{cancel}
\usetikzlibrary{arrows}
\usepackage{lipsum} 
\renewcommand*{\backref}[1]{}
\renewcommand*{\backrefalt}[4]{%
	\ifcase #1 (Not cited.)%
	\or        (Cited on page~#2.)%
	\else      (Cited on pages~#2.)%
	\fi}
\usepackage{caption}
\usepackage{adjustbox}
\usepackage[T1]{fontenc}
\usepackage[utf8]{inputenc}
\usepackage[french,english]{babel}

\newtheorem{theorem}{Theorem}
\newtheorem{lemma}[theorem]{Lemma}
\newtheorem{proposition}[theorem]{Proposition}

\newtheorem{question}[theorem]{Question}

\theoremstyle{remark}

\newcommand{\bF}{\mathbb{F}}

\newcommand{\bQ}{\mathbb{Q}}
\newcommand{\bP}{\mathbb{P}}
\newcommand{\bR}{\mathbb{R}}
\newcommand{\bZ}{\mathbb{Z}}

\newcommand{\cE}{\mathcal{E}}

\newcommand{\cO}{\mathcal{O}}

\newcommand{\tC}{\tilde{C}}
\newcommand{\tE}{\tilde{E}}
\newcommand{\tZ}{\tilde{Z}}

\newcommand{\lcm}{{\rm l.c.m.}}

\newcommand{\mult}{{\rm mult}}

\newcommand{\NE}{{\overline{\rm NE}}}
\newcommand{\Nef}{{\rm Nef}}

\newcommand{\Pic}{{\rm Pic}}

\newcommand{\Supp}{{\rm Supp}}

\newcommand{\Vol}{{\rm Vol}}

\def\<{\langle}
\def\>{\rangle}

\def\ra{\rightarrow}
\def\dra{\dashrightarrow}

\author[C.-J.~Lai]{Ching-Jui~Lai} 
\email{cjlai72@mail.ncku.edu.tw}
\address[]{Department of Mathematics, National Cheng Kung University, Tainan 70101, Taiwan
}

\author[T.-J.~Lee]{Tsung-Ju Lee}
\email{tsungju@gs.ncku.edu.tw}
\address[]{Department of Mathematics, National Cheng Kung University, Tainan 70101, Taiwan
}

\title{A Supplement to the anticanonical Volumes of weak $\bQ$-Fano threefolds of Picard rank two}
\date{\today}

\begin{document} 
\begin{abstract} We show that for a weak $\bQ$-Fano threefold $X$ ($\bQ$-factorial with terminal singularities and $-K_X$ is nef and big) of Picard rank $\rho(X)\leq 2$, either $-K_X^3\leq 64$ or $-K_X^3=72$ and $X=\bP_{\bP^2}(\cO_{\bP^2}\oplus\cO_{\bP^2}(3))$. This is supplementary to the 
previous work in \cite{L}.
\end{abstract}
\maketitle

A normal projective variety $\widetilde{X}$ is \emph{canonical Fano} if $-K_{\widetilde{X}}$ is an ample $\bQ$-Cartier divisor and $\widetilde{X}$ has at worst canonical singularities, i.e.~${\rm mult}_E(K_Y-f^*K_{\widetilde{X}})\geq0$ for any resolution $f\colon Y\rightarrow \widetilde{X}$ and any prime divisor $E$ on $Y$. Fano varieties are building blocks of uniruled varieties, and canonical singularities are indispensable in studying higher-dimensional geometry, cf.~\cites{YPGC, KM}.  A normal projective variety $X$ is called a 
\emph{weak $\bQ$-Fano variety} if it is $\bQ$-factorial with at worst terminal singularities, i.e.~${\rm mult}_E(K_Y-f^*K_X)>0$ for any prime divisor $E$ on $Y$ that is exceptional over $X$, and the anticanonical divisor $-K_X$ is nef and big. A canonical Fano variety $\widetilde{X}$ always has a weak $\bQ$-Fano variety $X$ as a crepant model by taking a $\bQ$-factorization and then a terminalization
\cite{BCHM}. Vice versa, a weak $\bQ$-Fano variety $X$ is a crepant model of its anticanonical model $\widetilde{X}:={\rm Proj}\left(\bigoplus_{m\geq0}\mathrm{H}^0(X,\cO_X(-mK_X))\right)$, which is a canonical Fano variety. Here being a crepant model means that there exists a proper birational morphism $\varphi\colon X\ra \widetilde{X}$ such that $K_X=\varphi^*K_{\widetilde{X}}$. In particular, the anticanonical volume satisfies $\Vol(\widetilde{X}):=(-K_{\widetilde{X}})^n=(-K_X)^n=:\Vol(X).$

It is known that the set of canonical Fano varieties of a fixed dimension is bounded \cites{Bir1, Bir2}  and hence so are their anticanonical volumes. In this article, we consider the problem of the effective bounds of $\Vol(X)$ for weak $\bQ$-Fano threefolds, which is equivalent to the same problem for canonical Fano threefolds and relevant for its classification \cites{Pr1, Pr2}. There exist explicit lower bound \cite{CC} and upper bound \cite{JZ24}, 
$$\frac{1}{330}\leq-K_X^3\leq 324,$$
where the first inequality is optimal by considering a general weighted hypersurface $X=X_{66}\subseteq\bP(1,5,6,22,33)$. The upper bound, though, possibly can be improved. 

\begin{question}\label{mainq}\cites{Pr1, L} 
The anticanonical volume of a weak $\bQ$-Fano threefold satisfies $\Vol(X)=-K_X^3\leq72.$ 
\end{question}

When $X$ is canonical Fano and Gorenstein, i.e. $K_X$ is Cartier,
then it is known that $-K_X^3\leq 72$ and the equality holds 
only for $X=\bP(1,1,1,3)$ or $\bP(1,1,4,6)$ by \cite{Pr1}. 
The approach of Prokhorov is to consider the geometry of a $K_X$-negative extremal ray, which always exists for a weak $\bQ$-Fano variety, and utilize the properties of the anticanonical system $|- K_X|$ of \emph{Cartier divisors} if necessary. Following Prokhorov, 
the first named author has formulated a streamlined program to tackle Question \ref{mainq} (cf.~\cite{L}). Let us recall the notation in \cite{Pr1} to elaborate the details. 
An extremal ray \(R\) on a normal projective variety \(X\) determines
an extremal contraction \(\varphi_{R}\). Set
\begin{equation*}
    \mathrm{Ex}(R):=\bigcup_{[C]\in R}C\subset X.
\end{equation*}
We say that a contractible extremal ray \(R\) is
of type \((n,m)\) if \(\dim\mathrm{Ex}(R) = n\) and
\(\dim \varphi_{R}(\mathrm{Ex}(R))=m\).
It is said to be of type \((n,m)^{+}\), \((n,m)^{-}\), or \((n,m)^{0}\)
if in addition the intersection number \(K_{X}\cdot R\) is positive, negative, or zero.

Suppose that $\varphi=\varphi_R\colon X\rightarrow W$ is a $K_X$-negative extremal contraction. 
There are four possibilities.
\begin{enumerate}
    \item $\dim W=0$. Then $X$ is a terminal $\bQ$-Fano threefold. If $X$ is Gorenstein, 
    then $X$ is smoothable to a smooth Fano threefold \cite{Nam} and $-K_X^3\leq 64$ 
    from the Mori--Mukai classification \cites{MM}. Otherwise, 
    $-K_X^3\leq 125/2$ and the equality holds only when $X=\bP(1,1,1,2)$ by \cite{Pr2}.
    
    \item $\dim W=1$. We have $\rho(X)=2$ and $-K_X^3\leq 54$ 
    when $X$ is Gorenstein by \cite{Pr1}*{Proposition 4.9}. For non-Gorenstein \(X\), 
    the first named author established the inequality $-K_X^3\leq 72$ except for one case \cite{L}.
    We affirm the same upper bound for this remaining case in this work, cf.~subsection \ref{MR}.

    \item $\dim W=2$. Then $X\ra W$ is a conic bundle and $W$ is a weak del Pezzo 
    surface with at worst type $A$ singularities \cites{MP, L}. 

    \item $\dim W=3$. If $\varphi$ is of type $(2,0)^-$, 
    then $W$ remains a weak $\bQ$-Fano threefold with $\rho(W)<\rho(X)$ 
    and $-K_W^3\leq -K_X^3$. We replace $X$ with $W$ and proceed. 
    If $\varphi$ is of type $(2,1)^-$, then $W$ is only ``almost Fano.'' 
\end{enumerate} 

When $\varphi$ is a conic bundle or is of $(2,1)^-$ type, to tackle Question \ref{mainq}, one may further investigate the (pluri-)anticanonical systems, and birational modifications of $W$ as in \cites{Pr1, Pr2}. 

\subsection{The Main Result}\label{MR} 
For a locally free sheaf \(\mathcal{E}\) on  a variety $X$, we can form the projective space bundle $\mathbb{P}_{X}(\mathcal{E}):=\mathrm{Proj}_{\mathcal{O}_{X}}\left(\mathrm{Sym}^{\bullet}\mathcal{E}\right)\to X$
over the base \(X\); the variety 
\(\mathbb{P}_{X}(\mathcal{E})\) is the projectivization of the total space 
of the vector bundle \(V_{\mathcal{E}}^{\vee}\) whose sheaf of sections is \(\mathcal{E}^{\vee}\).

\begin{theorem}[=Theorem \ref{thm:main}]\label{mainthm} Let $X$ be a weak $\bQ$-Fano threefold of Picard rank $\rho(X)=2$. Either $-K_X^3\leq64$ or $-K_X^3=72$ and 
$X\cong\bP_{\bP^2}(\cO_{\bP^2}\oplus\cO_{\bP^2}(3))$.
\end{theorem}

We remark here Theorem \ref{mainthm} covers case $(2)$ and partially case $(3)$ as discussed in the introduction. The proof of Theorem \ref{mainthm} utilizes the results of running two-ray games associated with two extremal rays of $X$ and a dichotomy of the resulting geometry. 
We have the following two cases:
\begin{itemize}
    \item Case {\bf (I)}.~$X$ equips with two Mori fibre spaces 
    and $-K_X^3\leq 54$, see \cite{L}*{Section 4}. 
    \item Case {\bf (II)}.~There is the diagram
\begin{center}
\begin{tikzcd}
			X_l\arrow[d,"f_l"] &&& X\arrow[lll,dashed,swap,"\chi_l"]\arrow[r,"\varphi_r"]   &Y_r\\
			Z_l   &&&E\arrow[u,hook]\arrow[r]&\varphi_r(E)\arrow[u,hook],
		\end{tikzcd} 
\end{center}
where
	\begin{enumerate}[(i)]
		\item $\varphi_r$ is a $K_X$-trivial divisorial contraction;
		\item $\chi_l$ either is the identity map 
or consists of $K$-flips; 
		\item $f_l$ is a Mori fiber space, i.e.~$-K_{X_l}$ is $f_l$-ample with $\rho(X_l/Z_l)=1$.
	\end{enumerate}
\end{itemize}
Our main subject is to show that $-K_X^3\leq 72$ in Case \textbf{(II)}. 

\subsection*{Acknowledgments}  
This work was partially completed when the first author visited the Center for Complex Geometry of the Institute for Basic Science. He thanks Prof.~Yongnam Lee for the invitation and the institute's hospitality. The first named author is supported by NSTC 111-2115-M-006-002-MY2. The second named author is supported by NSTC 112-2115-M-006-016-MY3.

\section{Preliminaries}\label{prem}
We adhere to the notation presented in \cite{L} and to the results established therein. 
Here and after, we assume that $X$ is in Case {\bf (II)}.

\begin{lemma}[\cite{L}*{Lemma 5.1}]
\label{adj} Suppose that $\dim \varphi_r(E)=0$. Then there is a curve $\Gamma$ 
on $E$ moving in a one-dimensional family such that $0<-E\cdot \Gamma\leq3.$ 
\end{lemma}

Before stating the other lemmas, let us recall the notation in \cite{L}. Consider
the diagram in Case {\bf (II)}.
\begin{itemize}
    \item If \(\dim Z_{l}=1\), we set \(F_{l}\) to be a general fiber of \(f_{l}\colon X_{l}\to Z_{l}\).
    \item If \(\dim Z_{l}=2\), let \(H_{l}\) be the pullback of 
an ample divisor on \(Z_{l}\). It is from \cite{L}*{Lemma 4.3} that 
\(H_{l}\) can be chosen such that \(K_{X_{l}}^{2}\cdot H_{l}\le 12\).
    \item Denote by \(E_{l}\) the proper transform of \(E\) on $X_l.$
\end{itemize}
The next lemma will be useful later.

\begin{lemma}[\cite{L}*{Lemma 5.2}]\label{gen} 
If $\dim Z_l=1$, then the divisors $F_{l}$ and $E_{l}$ are linearly independent in $N^1(X_{l})_\bR$. 
Moreover, we can write $-K_{X_{l}}\equiv aF_{l}+bE_{l}$ with $a,b\in\bQ_{>0}$ and $0< b\leq 3$. 
Similarly, when $\dim Z_l=2$ we can write $-K_{X_{l}}\equiv aH_{l}+bE_{l}$ with $a,b\in\bQ_{>0}$ and $0< b\leq 2$. In either cases, we have 
$$\begin{cases}0<a\leq2b,&\ {\rm if}\ \dim\varphi_r(E)=1;\\ 0<a\leq 3b,&\ {\rm if}\ \dim\varphi_r(E)=0.\end{cases}$$ 
\end{lemma}

\begin{lemma}\label{Volbd} 
Suppose that $X$ is a weak $\bQ$-Fano threefold of Picard rank two and we are in Case {\bf (II)}. Then $-K_X^3\leq72$ unless $\dim Z_l=1$ and $\dim\varphi_r(E)=0$.
\end{lemma}
\begin{proof} The argument in \cite{L}*{Section 5} shows that 
$$-K_X^3\leq\begin{cases} aK_{X_l}^2\cdot F_l=aK_{F_l}^2\leq 9a,&\ {\rm if}\ \dim Z_l=1;\\  aK_{X_l}^2\cdot H_l\leq 12a, &\ {\rm if}\ \dim Z_l=2,\end{cases}$$
where in the last inequality we have used \cite{L}*{Lemma 4.3}. Now the lemma follows immediately from Lemma \ref{gen}, 
$$-K_X^3\leq\begin{cases} 9\cdot6=54,  &\ {\rm if}\ \dim Z_l=1,\ \dim\varphi_r(E)=1;\\  9\cdot9=81 ,&\ {\rm if}\ \dim Z_l=1,\ \dim\varphi_r(E)=0;\\ 12\cdot4=48 ,&\ {\rm if}\ \dim Z_l=2,\ \dim\varphi_r(E)=1;\\ 12\cdot6=72, &\ {\rm if}\ \dim Z_l=2,\ \dim\varphi_r(E)=0.\end{cases}$$ 
\end{proof}
Note that when $\dim Z_l=1$, we have $Z_l=\bP^1$ as $X$ is rationally connected. 

\section{Case {\bf (II)} with \texorpdfstring{$Z_l=\bP^1$}{} and \texorpdfstring{$\dim\varphi_r(E)=0$}{}}\label{CaseII}
\label{caseII}
In this section, we assume that $X$ is a weak $\bQ$-Fano threefold of Picard rank two and suppose we are in Case {\bf (II)}. 
Furthermore, we assume that $Z_{l}=\mathbb{P}^{1}$ and $\dim\varphi_{r}(E)=0$. 

\begin{lemma}\label{key0} Assume $X$ is within the above setup. If $-K_X^3>72$, then 
\begin{enumerate}[$(a)$]
	\item $-K_{X_l}\equiv aF_{l}+3E_l$, where $8<a\leq9$, $F_{l}\cong\bP^2$, and $E_l|_{F_{l}}\cong\cO_{\bP^2}(1)$;
	\item $\chi\colon X\dra X_l$ can not be the identity. 
\end{enumerate}
\end{lemma}
\begin{proof} 
For simplicity, we drop the subscript \(l\) and simply write 
$F:=F_{l}$. 

A general fiber $F$ of the Mori fiber space $f_l\colon X_l\rightarrow Z_l=\bP^1$ is a smooth del Pezzo surface and hence $K_F^2\leq 9$, where equality holds only when $F\cong\bP^2$. Suppose that either $K_F^2\leq8$ or $b\leq2$. As $a\leq 3b\leq9$ by Lemma \ref{gen}, it follows from the proof of \cite{L}*{Proposition 5.5} that
$$-K_X^3\leq aK_F^2\leq \begin{cases}9\cdot8=72, &\ {\rm if}\ K_F^2\leq8, \\ 6\cdot9=54,&\ {\rm if}\ b\leq2. \end{cases}$$ 
Hence we assume that $K_F^2>8$ and $b>2$. This is only possible when $K_F^2=9$ and $F\cong\bP^2$. Immediately we have $a\in(8,9]$ from 
$72<-K_X^3\leq aK_F^2=9a.$

Note that a line $C_l\subseteq F=\bP^2$ generates the corresponding extremal ray of the contraction $f_l$. As $E_l|_F$ is effective and Cartier, we get $E_l\cdot C_l=E_l|_F\cdot C_l\in\bZ_{\geq0}$. If $E_l\cdot C_l=0$, then $E_l=f_l^*(D)$ for some effective non-zero divisor $D$ on $\bP^1$. This is a contradiction since 
$$1=\kappa(\bP^1,D)=\kappa(X_l,E_l)=\kappa(X,E)=0,$$
where the last equality holds as $E$ is $\varphi_r$-exceptional. Hence $E_l\cdot C_l\in\bZ_{>0}$. Since $-K_{X_l}\equiv aF+bE_l$ and $F\cdot C_l=0$, we have  
$$2<b\leq bE_l\cdot C_l=-K_{X_l}\cdot C_l=-K_F\cdot C_l=3.$$
This implies that $E_l\cdot C_l=1$ and $b=3$. It follows that $-K_{X_l}\equiv aF+3E_l$, $-K_F=-K_{X_l}|_F\equiv3E_l|_F$, and  $E_l|_F\cong\cO_{\bP^2}(1)$. This finishes the proof of $(a)$.

For $(b)$, if $\chi\colon X\dra X_l$ is the identity, then $F^2\equiv0$ and 
$$-K_X^3=(aF+3E)^3=27a+27E^3>81.$$
where the last inequality follows from $a>8$ in $(a)$ and $E^3>0$. This contradicts Lemma \ref{Volbd}. 
\end{proof}

\begin{proposition}\label{key1} 
With the above setup and the assumption that $-K_X^3>72$, the del Pezzo fibration $f_l\colon X_l\rightarrow \bP^1$ cannot be a $\bP^2$-bundle.
\end{proposition}
\begin{proof}
Suppose on the contrary that $f_l\colon X_l\rightarrow\bP^1$ is a $\bP^2$-bundle. 
By Lemma \ref{key0}, the Cartier divisor $E_l$ satisfies $E_l|_{f^{-1}(p)}\cong\cO_{\bP^2}(1)$ for 
all $p\in\bP^1.$ By cohomology and base change \cite{H}*{Theorem 12.11}, 
$\cE:=(f_l)_*\cO_{X_l}(E_l)$ is locally free of rank 3 and $X_l\cong\bP(\cE)$. 
By Grothendieck's theorem \cite{OSS}, the locally free sheaf $\cE$ is a direct sum of line bundles. 
Taking a normalization, we may assume that $\cE=\cO_{\mathbb{P}^{1}}(\alpha)\oplus
\cO_{\mathbb{P}^{1}}(\beta)\oplus\cO_{\mathbb{P}^{1}}$ for 
some integers $\alpha\geq\beta\geq0$. Let $\cO_{X_{l}}(1)$ be the dual of
the tautological bundle of $\mathbb{P}(\cE)$. 
By abuse of notation, we shall also denote by 
\(\cO_{X_{l}}(1)\) the divisor associated to the line bundle \(\cO_{X_{l}}(1)\).
Since $\Pic(X_l)=\bZ F+\bZ\cO_{X_{l}}(1)$, 
we can write $\cO_{X_{l}}(E_l)\cong\cO_{X_{l}}(1)\otimes \cO_{X_{l}}(tF)$ 
for some $t\in\bZ$.  
By the projection formula, 
\begin{align*} 1&=h^0(Y_r,\cO_{Y_r})=h^0(X,\cO_X(E))\\
&=h^0(X_l,\cO_{X_l}(E_l))=h^0(\bP^1,\cO_{\bP^1}(t+\alpha)\oplus\cO_{\bP^1}(t+\beta)\oplus\cO_{\bP^1}(t)),
\end{align*}
so we get $-t=\alpha>\beta\geq0.$ On the other hand, from the Euler sequence and Lemma \ref{key0} $(a)$,
$$-K_{X_l}=\cO_{X_{l}}(3)+(2-\alpha-\beta) F\equiv aF+3E_l\sim \cO_{X_{l}}(3)+(a-3\alpha)F.$$
It follows that $a=2+3\alpha-(\alpha+\beta)=2+2\alpha-\beta\in\bZ$ and hence $a=9$ from Lemma \ref{key0} $(a)$. 
Therefore, from the assumption
$$0\leq 2\alpha-7=\beta<\alpha,$$
we see that $(\alpha,\beta)=(6,5),$ $(5,3)$, or $(4,1)$ are the only 
possible solutions.

Following \cite{CLS}*{Example 7.3.5}, we employ a toric description for the 
toric variety $X_l=\bP_{\bP^1}(\cO_{\bP^1}\oplus\cO_{\bP^1}(\alpha)\oplus\cO_{\bP^1}(\beta))$. 
Consider in $\bR^3$ the vectors 
$$-u_0=u_1=(1,0,0),~e_1=(0,1,0),~e_2=(0,0,1),~e_0=-e_1-e_2.$$
Let $v_0=u_0+\alpha e_1+\beta e_2$, $v_1=u_1$, and consider the cones 
\begin{align*} 
\sigma_0^+=\<v_1,e_1,e_2\>;~\sigma_1^+=\<v_1,e_0,e_2\>;~\sigma_2^+=\<v_1,e_0,e_1\>;\\
\sigma_0^-=\<v_0,e_1,e_2\>;~\sigma_1^-=\<v_0,e_0,e_2\>;~\sigma_2^-=\<v_0,e_0,e_1\>.
\end{align*} 
Then fan, which is the collection of the cones above together
with all their proper faces, defines the toric variety \(X_{l}\).

If $D_i$ and $E_j$, $i=0,1,2$ and $j=0,1$, are toric divisors 
corresponding to $e_i$ and $v_j$ respectively, then 
\begin{align*} D_0\stackrel{\chi^{e_1^\vee}}{\sim}D_1+\alpha 
E_0,\ D_0\stackrel{\chi^{e_2^\vee}}{\sim}D_2+
\beta E_0,\ E_0\stackrel{\chi^{u_1^\vee}}{\sim}E_1,
\end{align*}
and
$$-K_{X_l}=D_0+D_1+D_2+E_0+E_1\sim 3D_0+(2-\alpha-\beta)E_1.$$
It is clear that $\cO(D_0)$ is the relative ample sheaf
and $F\sim E_0\sim E_1$ is the fiber class of $f_l\colon X_l\rightarrow Z_l$.

There are, in total, nine toric invariant curves described as, for $0\leq i\leq 2$,  
$$C_i=V(\tau(v_1e_i));~C'_i=V(\tau(v_0e_i));~C''_i=V(\tau(e_0\hat{e_i}e_2)),$$
which satisfy 
$$F\cdot C_i=F\cdot C_i'=0,\ F\cdot C''_i=1;$$ 
$$D_0\cdot C_i=D_0\cdot C'_i=1,~D_0\cdot C''_0=0,~D_0\cdot C''_1=\alpha,~D_0\cdot C''_2=\beta.$$
We infer that both $D_0$ and $F$ are nef and hence base point free (cf.~\cite{CLS}). 
Clearly, we have $[C_l]=[C_i]=[C'_i]$, so the nef cone and the Mori cone of curves are respectively given by
$$\Nef(X_l)=\bR_{\geq0}[D_0]+\bR_{\geq0}[F];\ \NE(X_l)=\bR_{\geq0}[C_l]+\bR_{\geq0}[C''_0].$$
Note that 
$$-K_{X_l}\cdot C_0''=2-\alpha-\beta<0.$$ 
By Lemma \ref{key0} $(b)$, the map $\chi_l\colon X\dashrightarrow X_l$ is not 
the identity map and it is a composition of finitely many flips. 
Hence, the contraction $g^{+}:=\varphi_{|mD_0|}\colon X_l^+:=X_l\rightarrow T$ must be the last anti-flipping contraction with $R=\bR_{\geq0}[C_0'']$ the contracting extremal ray. 

The flipping contraction $g^{-}\colon X_l^-\rightarrow T$ can be constructed by 
replacing the wall $\tau(e_1e_2)$ with $\tau(v_0v_1)$ so that the cones of $X_l^-$ are 
\begin{align*} 
\tilde{\sigma}_0^+=\<v_0,v_1,e_1\>;~\sigma_1^+=\<v_1,e_0,e_2\>;~\sigma_2^+=\<v_1,e_0,e_1\>;\\
\tilde{\sigma}_0^-=\<v_0,v_1,e_2\>;~\sigma_1^-=\<v_0,e_0,e_2\>;~\sigma_2^-=\<v_0,e_0,e_1\>.
\end{align*} 
See Figure \ref{flip}.
\begin{center}
\begin{figure}[t!]
  \caption{The last flip.}
\includegraphics[scale=0.65]{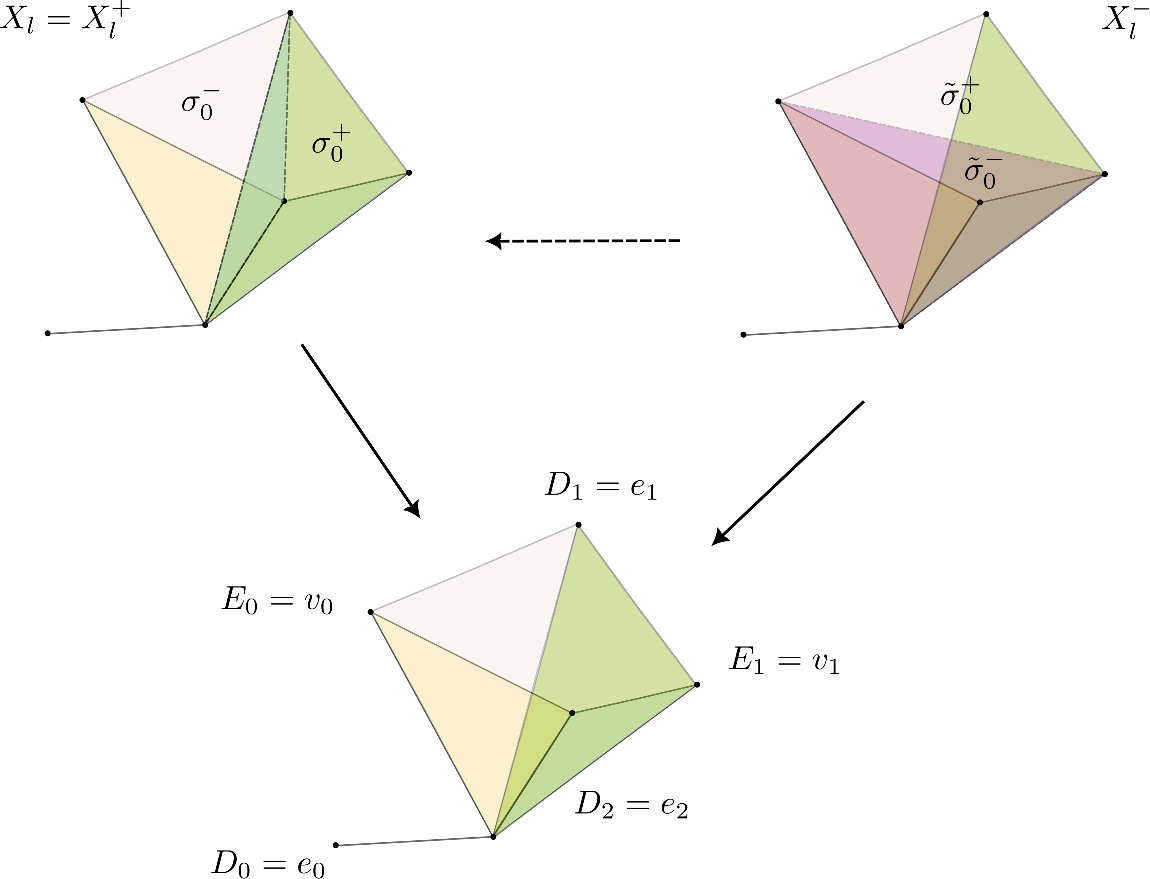}
 \label{flip}
\end{figure}    
\end{center}

Indeed, to avoid introducing new notations, 
denote by $D_i$, $E_j$ again the toric divisors corresponding to 
\(v_{i}\) and \(e_{j}\)
on $X_l^-$ and set $C_-:=V(\tau(v_0v_1))$. We have $D_0\cdot C_-=0$, 
$$D_1\cdot C_-=\frac{\mult(\tau(v_0v_1))}{\mult(\tilde{\sigma}_0^+)}=
\frac{\lcm(\alpha,\beta)}{\beta},$$
and
$$D_2\cdot C_-=\frac{\mult(\tau(v_0v_1))}{\mult(\tilde{\sigma}_0^-)}=\frac{\lcm(\alpha,\beta)}{\alpha}.$$

From the wall relation $\alpha e_1-v_0-v_1+\beta e_2=0$, we also get 
\begin{align*}E_0\cdot C_-&=E_1\cdot C_-=\frac{-1\cdot\mult(\tau(v_0v_1))}{\alpha\cdot\mult(\tilde{\sigma}_0^+)}=\frac{-1\cdot\mult(\tau(v_0v_1))}{\beta\cdot\mult(\tilde{\sigma}_0^-)}=-\frac{1}{\alpha\beta}.
\end{align*}
Since  $-K_{X^-_l}=E_0+E_1+D_0+D_1+D_2$, this implies that
$$-K_{X_l^-}\cdot C_-=\frac{\alpha+\beta-2}{\alpha\beta}>0.$$

Note that $X_l^-$ has two singular points along two maximal cones $\{\tilde{\sigma}_0^+, \tilde{\sigma}_0^-\}$. 
Consider $m_{\tilde{\sigma}_0^-}=(1,(2-\beta)\alpha^{-1},1)$ on $\tilde{\sigma}_0^-$ 
the associated $\bQ$-Cartier data of the canonical divisor.  For all possible solutions 
$(\alpha,\beta)=(6,5), (5,3),$ or $(4,1)$, the discrepancy $a(E_w)$ of the exceptional
divisor \(E_{w}\) obtained by adding the \(1\)-cone
generated by the primitive generator $w=(0,1,1)\in{\rm int}(\tilde{\sigma}_0^-)$ 
is 
$$a(E_w)=m_{\tilde{\sigma}_0^-}\cdot(0,1,1)-1=
\frac{2-\beta}{\alpha}+1-1=\frac{2-\beta}{\alpha}<0.$$
Hence, $X_l^-$ is not terminal. On the other hand, we know that all 
the intermediate models appearing in the two-ray game $\chi_l\colon X\dra X_l$ should be terminal. This is absurd. 
\end{proof}

\begin{theorem} Let $X$ be a weak $\bQ$-Fano threefold of Picard rank two. If we are in Case {\bf(II)}, then $-K_X^3\leq 72$. 
\end{theorem}
\begin{proof} By Lemma \ref{Volbd}, it is enough to rule out the case that when $\dim Z_l=1$, $\dim\varphi_r(E)=0$, and $-K_X^3>72$.

Consider the projective bundle $X':=\bP((f_l)_*\cO_{X_l}(E_l))$ over $Z_l=\bP^1$. The associated map $g\colon X_l\dashrightarrow X'$ is birational since $X_l$ is generically smooth over $Z_l$ and $E_l|_{F_l}=\cO_{\bP^2}(1)$ by Lemma \ref{key0} $(a)$. Moreover, $g$ contracts no divisors as $\rho(X)=\rho(X')=2$ and generically an isomorphism on $F_l$ and $E_l$. Hence $g$ is small. By Proposition \ref{key1}, $g$ cannot be an isomorphism. 
By the base point freeness theorem \cite{KM}, a $\bQ$-complement $K_X+D\sim_\bQ0$ exists with $(X,D)$ being terminal. Since $X\dra X_l$ and  $X\dra X'$ are small, $(X_l,D_l)$ and $(X',D')$ are also terminal log Calabi--Yau pairs. By \cite{K:flop}, $X_l\dra X'$ factors into a sequence of flops. However, this contradicts the last part of the proof of Proposition \ref{key1} since all intermediate models should be terminal. 
\end{proof}

\section{Large Volume case}\label{LargeV}
It is known that if $X$ is $K$-semistable, then $-K_X^3\leq 64$ by \cite{Liu:Vol}. 
Then it is a curious problem to see what causes a Fano variety to be $K$-unstable. 
We exploit this question by investigating weak $\bQ$-Fano threefolds of Picard rank 
two with anticanonical degree $-K_X^3>64$. Note that this classification is complete if $X$ is Gorenstein without any constraint on the Picard rank by \cites{Kar1, Kar2}.

If $X$ is a weak $\bQ$-Fano threefold of $\rho(X)=2$ and $\varphi\colon X\ra X'$ is a $K$-negative contraction, either 
\begin{enumerate}
    \item $\varphi$ is divisorial so that $-K_X^3\leq -K_{X'}^3\leq 64$, as $X'$ is terminal Fano of $\rho(X')=1$ and we use \cite{Pr2}, or  
    \item we are in Case ${\bf (I)}$ or ${\bf (II)}$ as detailed in \cite{L}, where we can improve the upper bound to \(64\) when \(X\) is not isomorphic to \(\bP_{\bP^2}(\cO_{\bP^2}\oplus\cO_{\bP^2}(3))\).
\end{enumerate}

\begin{theorem}\label{thm:main} Let $X$ be a weak $\bQ$-Fano threefold 
with \(\rho(X)=2\) and $-K_{X}^{3}>64$. Then \(X\cong\bP_{\bP^2}(\cO_{\bP^2}\oplus\cO_{\bP^2}(3)))\).
\end{theorem}

\begin{proof}
From \cite{L} and the proof of Lemma \ref{Volbd}, we know that 
$-K_X^3>64$ 
only 
can occur in Case ({\bf II}).
We divide the theorem into two cases according to $\dim Z_l$. 

Recall from Lemma \ref{gen} that we can write 
\begin{equation}
    \begin{cases}
        -K_{X_{l}}\equiv aF_{l}+bE_{l},~&\mbox{if}~\dim Z_{l}=1,\\
        -K_{X_{l}}\equiv aH_{l}+bE_{l},~&\mbox{if}~\dim Z_{l}=2.
    \end{cases}
\end{equation}
For notations, see the paragraph before Lemma \ref{gen}.
In those cases, we have 
$$0<a\leq 3b\leq \begin{cases}9\ &{\rm if} \dim Z_l=1,\\6\ &{\rm if} \dim Z_l=2.\end{cases}$$
Below we write $Z=Z_l$.

\subsection{\texorpdfstring{$\dim Z=2$}{}} 
It follows that the base $Z$ is a weak del Pezzo surface with at worst type \(A\) singularities
(cf.~\cite{L}*{Theorem 2.2 and Proposition 2.3}). Moreover, by considering $H_l:=f_l^*(C)$ 
the pull-back of an extremal curve $C$ of minimal length $-K_Z \cdot C\leq 3$, we have as in the proof of \cite{L}*{Lemma 4.3} that 
\begin{equation}
    64< -K_X^3\leq a \cdot K_{X_l}^2\cdot H_l=a\cdot (f_l)_*K_{X_l}^2\cdot C\leq 6\cdot 12=72.
\end{equation}
Since $0<a\leq 3b\leq6$, this leads to 
\begin{equation}
    \frac{64}{6}<K_{X_l}^2\cdot H_l=4(-K_Z\cdot C)-\Delta\cdot C\leq 12.
\end{equation}
If $\mu\colon\tZ\rightarrow Z$ is the minimal resolution of $Z$, then there is an extremal curve $\tC$ on $\tZ$ such that 
$\mu|_{\tC}\colon\tC\rightarrow C$ is birational and $-K_Z\cdot C=-K_{\tZ}\cdot\tC\in\bZ_{>0}.$ 
Recall that \(\Delta\) is the discriminant divisor for \(f_{l}\). Now there are two possibilities.

\quad \\
\noindent{\bf Case 1: $C\nsubseteq\Supp(\Delta)$.} 
If $C\nsubseteq\Supp(\Delta)$, then we have
\begin{equation}
    \frac{64}{6}+\Delta\cdot C<4(-K_{\tZ}\cdot\tC)\leq 12.    
\end{equation}
This implies that $-K_{\tZ}\cdot\tC=3$ and $\tZ\cong\bP^2$. 
In particular, we have $Z\cong\tZ\cong\bP^2$, 
$C\cong\tilde{C}\cong\mathbb{P}^{1}$ is a line, and
$\Delta\cdot C\in\{0,1\}$. 
If $\Delta\cdot C\ne 0$, then $\Delta\ne\emptyset$ and \(\Delta\)
must be a smooth rational curve in \(\mathbb{P}^{2}\).
Since \(\Delta\) is a tree of rational curves, by \cite{Pr1}*{Lemma 5.3},
we deduce that $\Delta\cdot C=0$, $\Delta=0$, and $f_{l}\colon 
X_l\rightarrow Z=\bP^2$ is a $\bP^1$-fibration. Moreover, 
because \(f_{l}\) is equi-dimensional,
\(X_{l}\) is Cohen--Macaulay, and \(\mathbb{P}^{2}\) is smooth,
it follows that \(f_{l}\) is a \emph{flat} morphism.
Since $K_{X_l}^2\cdot H_l=12$, the inequalities
\begin{equation}
    \frac{64}{12}=\frac{16}{3}\leq a\leq 3b\leq6
\end{equation} 
imply that $b=2$ and \(a=6\). 

As in the proof of Proposition \ref{key1}, 
we denote by \(E_{l}\subset X_{l}\)  
the proper transform of \(E\) under \(\chi\). We can check that
\(E_{l}\cdot f^{-1}_{l}(p)=1\) for every \(p\in\mathbb{P}^{2}\).
It follows by cohomology and base change (recall that \(f_{l}\) is flat)
that \(f_{l\ast}\mathcal{O}_{X_{l}}(E_{l})\) 
is a rank two vector bundle over \(\mathbb{P}^{2}\).
There is also a surjection
\begin{equation}
    f_{l}^{\ast}f_{l\ast}\mathcal{O}_{X_{l}}(E_{l})\to \mathcal{O}_{X_{l}}(E_{l}).
\end{equation}
Thus \(X_{l}\cong \mathbb{P}_{\mathbb{P}^{2}}(\mathcal{E})\) with 
\(\mathcal{E}=f_{l\ast}\mathcal{O}_{X_{l}}(E_{l})\).
So \(X_{l}\) is smooth and there cannot be
any small contraction starting from \(X_{l}\). Therefore, the birational morphism
\(\chi\) in the diagram must be the identity map and
\(X\) itself is the projectivization of a rank two vector bundle over \(\mathbb{P}^{2}\).
Now we are in the scenario considered in \cite{Pr1}*{\S5} and hence we 
conclude that 
\(X=\mathbb{P}_{\mathbb{P}^{2}}(\mathcal{O}\oplus 
\mathcal{O}(3))\) is the only possibility.

\quad \\
\noindent{\bf Case 2: $C\subseteq\Supp(\Delta)$.} 
If $C\subseteq\Supp(\Delta)$, then from \cite{L}*{Lemma 4.3},
\begin{align*}64&\leq -K_X^3=aK_{X_l}\cdot H_l^2\\
&=a(3(-K_{\tZ}\cdot\tC)+2-2p_a(\tC)-\tE\cdot\tC-C\cdot(\Delta-C))\\
&\leq a\cdot 11\leq 66,
\end{align*}
where $\mu\colon\tZ\rightarrow Z$ is the minimal 
resolution, $\tC$ is the proper transform of $C$, and $\mu^{\ast}C=\tC+\tE.$ 
Hence 
\begin{equation}
    \frac{64}{11}\leq a\leq 3b\leq 6.
\end{equation}
If $p_a(\tC)\geq1$ or $-K_{\tZ}\cdot\tC\leq2$, then $-K_X^3\leq a\cdot9\leq54$. 
It follows that $-K_{\tZ}\cdot \tC=3$, $\tZ=Z=\bP^2$, and $\tE=0$. 
Now, if $\Delta-C\neq0$, then $-K_X^3\leq a\cdot 10\leq 60$. 
Hence, $\Delta=C\cong\bP^1$ is a line in $\bP^2$. 
The same argument as in the previous case shows that $\Delta\cong\bP^1$ is impossible, 
so this instance does not occur.

\subsection{\texorpdfstring{$\dim Z_l=1$}{}} We can write $-K_X\equiv a F+bE$ with 
$0<a\leq 3b\leq 9$. From the proof of Lemma \ref{Volbd}, we have 
$$-K_X^3\leq aK_F^2\leq\begin{cases}9\cdot8=72, &\ {\rm if}\ K_F^2\leq8 \\ 6\cdot9=54,&\ {\rm if}\ b\leq2 \end{cases}.$$ 
Hence if $-K_X^3\geq64$, then $b>2$ and $K_F^2\leq8$. Moreover, $K_F^2=8$, 
for otherwise $-K_X^3\leq 9\cdot7=63$. We conclude that
the smooth del Pezzo surface $F=\bP^1\times\bP^1$ or $\bF_1$. 
By taking $C\in|\cO_{\bP^1\times\bP^1}(1,1)|$ or a fiber of $\bF_1\rightarrow\bP^1$, 
we always have $-K_F\cdot C=2.$ 
Since 
$$2=-K_F\cdot C=-K_X\cdot C=(aF+bE)\cdot C=b(E\cdot C)$$ 
and $E\cdot C\in\bZ_{>0}$, we find $b\leq2$ and this is a contradiction. Hence, $\dim Z_l=1$ does not occur. 
\end{proof}

\bibliographystyle{alpha}
\bibliography{vol}		

\end{document}